\documentclass{fundam}
\usepackage{eufrak}
\usepackage{euler}
\usepackage{latexsym}
\usepackage{amsmath}
\issue{97~(2009)}
\usepackage{amssymb}
\usepackage[inline]{asymptote}
\usepackage{proof}
\usepackage{tabularx}
\usepackage{url}

\newcommand{\ml}{\mathfrak{L}}
\begin{document}
\setcounter{page}{177}

\title{Algebraic Semantics of Similarity-Based Bitten Rough Set Theory}

\author{\textbf{A. Mani}\thanks{I would like to the thank the referee for useful remarks and the reference to \cite{SW3}}\\
Member, Calcutta Mathematical Society\\
9/1B, Jatin Bagchi Road\\
Kolkata(Calcutta)-700029, India\\
\texttt{Email: $a.mani@member.ams.org$}\\
Homepage: \url{http://www.logicamani.co.cc}}

\maketitle

\runninghead{A. Mani}{Algebraic Semantics of Bitten Rough Sets}

\begin{abstract}
We develop two algebraic semantics for bitten rough set theory (\cite{SW}) over similarity spaces and their abstract granular versions. Connections with choice based generalized rough semantics developed in \cite{AM69} by the present author and general cover based rough set theories are also considered.
\end{abstract}


\section{Introduction}

In classical rough set theory, information systems with equivalence relations (or approximation spaces) are studied. The data tables used have crisp symbolic values (see \cite{ZPB} for example). More general variants with other types of relations and values have also been extensively studied in the literature (see \cite{KM}, \cite{DO}, \cite{ZP5}, \cite{AM24}, \cite{SS1} for example). The methods and concepts, including those of definability, approximation, granular representability of approximations, neighbourhoods and others, need many modifications for application in these situations. We focus on similarity spaces or tolerance approximation spaces in this research paper. Tolerance approximation spaces are also of interest in representation of inexact and incomplete knowledge (see \cite{ZP5}, \cite{SW3}).

In classical rough set theory, the negative region of a set is the lower approximation of the complement of the set. This region is disjoint from the upper approximation of the set in question. An analogous property fails to hold in tolerance approximation spaces (TAS). To deal with this a different semantic approach (to TAS) involving modified upper approximations is proposed in \cite{SW}. These modified upper approximations are formed from upper approximations by 'biting off' a part of it to form 'bitten upper approximations'. The new approximations turn out to be disjoint from the negative region of the subset and also possess better properties.

More specifically, the upper approximation of a subset of a TAS formed by a set of granules (with respect to those on the subset) is reduced by the deletion of the negative region of the subset to form the 'bitten upper approximation' of the subset. This ensures that the new upper approximation is disjoint from the negative region of the subset. In the formation of approximations explicit constraints are not imposed on the possible type of granules. The authors also pose the problem of developing an algebraic semantics for the approach.

In the present paper, we develop two different algebraic semantics for the same. 
The first of these actually captures reasoning within the power set of the set of possible order-compatible partitions of the set of roughly equivalent objects. We also prove suitable representation theorems in the third section. The main result of \cite{BZ} is required in the proof of the same theorem. 

Between formalizing the interaction of roughly equivalent objects and formalizing the approximations in a classicalist setting, we prefer the former or a dialectical version thereof (\cite{AM699}). But whether this 'rough equivalence' is to be an equivalence relation or a congruence or something else, is an open question. We do not abandon the viewpoint that it should at least be an equivalence. Interestingly the known algebraic approaches (like \cite{CG98}) aimed at just the approximations and derived operations in a classicalist setting (as opposed to the domain of roughly equal objects) cannot be directly adapted (at least) to the present context. 

Given that it may not be possible to have suitable abstract representation theorems from partial algebras over the set of roughly equivalent objects in the context, we use choice functions in a crucial (though implicit) way to introduce \emph{simplified algebra of the bitten granular semantics} \textsf{SGBA}. These are developed up to the level of some solvable/open problems. Here the motivation is not to integrate choice functions with rough set theory as in \cite{AM69}, but to build an easier semantics. At the same time the semantics is also motivated by possible connections with choice inclusive similarity based rough set theory.

In the following section we mention some of the essential notions and outline the bitten approach. The first of the algebraic semantics for the bitten approach is developed in the third section. Some related problems are also posed in the same. In the fourth section, a partial algebraic semantics over a relatively easier semantic domain (with respect to the previous approach) is developed.
These semantics are also applicable to particular cases of other general rough set theories including the cover based ones. This is indicated in the fifth section.  

\section{Background}

By a \emph{Tolerance Approximation Space} (TAS), we mean a pair of the form $S\,=\,\left\langle \underline{S},\,T\right\rangle$, with $\underline{S}$ being a set and $T$ a tolerance relation over it - these are also known as similarity or tolerance spaces. Some references for extension of classical rough set theory to TAS are \cite{CG98}, \cite{IM}, \cite{SW3}, \cite{CC}, \cite{AM69} and \cite{KM}. These theories differ in the types of granules, process of definition of approximations, semantics and on computational aspects. We mention some important aspects of rough set theory over TAS before proceeding with the \emph{bitten proposal} in \cite{SW}.

An approach (\cite{KM}) has been to define a new equivalence $\theta_{0}$ on  $S$ via $(x,\,y)\,\in\,\theta_{0}$ if and only if $dom_{T}(x)\,=\,dom_{T}(y)$  with  $dom_{T}(z)\,=\,\cap\{[x]_{T}\,:\,z\,\in\,[x]_{T}\}$. This is essentially an unduly cautious 'clear perspective' approach.

In ~\cite{IM}, a theory of generalized rough sets based on covers of subsets of a given set $\underline{S}$ is considered. Let ${S}$ be a set and $\mathcal{K}\,=\,\{K_{i}\}_{1}^{n}\,:n\,<\,\infty$ be a collection of subsets of it. We will abbreviate subsets of natural numbers of the form $\{1, 2, \ldots , n\}$ by $\mathbf{N}(n)$. If $X\subseteq  S$, then consider the sets (with $K_{0}\,=\,\emptyset$, $K_{n+1}\,=\,S$ for convenience):  
\begin{enumerate} \renewcommand\labelenumi{\theenumi}
  \renewcommand{\theenumi}{(\roman{enumi})}
\item {$X^{l1}\,=\,\bigcup\{K_{i}\,:\,K_{i}\subseteq  X,\,\,i\,\in \,\{0,\,1,\,...,n\}  \}$} 
\item {$X^{l2}\,=\,\bigcup\{\cap_{i\,\in \,I}(S\,\setminus K_{i})\,:\,\cap_{i\,\in \,I}(S\,\setminus\, K_{i})\subseteq X,\,\,I\subseteq \mathbf{N}(n+1) \} $}
\item {$X^{u1}\,=\,\bigcap\{\cup_{i\,\in \,I}{K_{i}}\,:\,X,\subseteq \cup_{i\,\in \,I}\,K_{i},\,\,I\subseteq \mathbf{N}(n+1) \} $}
\item {$X^{u2}\,=\,\bigcap\{S\,\setminus\,{K_{i}}\,:\,X\subseteq S\,\setminus\,K_{i},\,\,i\,\in \,\{0,\,...,n\}\}$}
\end{enumerate}

The pair $(X^{l1},\,X^{u1})$ is called a $AU$-\emph{rough set} by union, while $(X^{l2},\,X^{u2})$ a $AI$-\emph{rough set} by intersection (in the present author's notation \cite{AM24}). In the notation of ~\cite{IM} these are $\left( \mathcal{F}_{*}^{\cup}(X),\,\mathcal{F}_{\cup}^{*}(X)\right)$ and $\left( \mathcal{F}_{*}^{\cap}(X),\,\mathcal{F}_{\cap}^{*}(X)\right)$ respectively. We will also refer to the pair $\left\langle S,\,\mathcal{K} \right\rangle $ as a \textsf{AUAI}-\emph{approximation system}. 

In the TAS context, let the granules be $[x]_{T}\,=\,\{y\,;\,(x,\,y)\,\in\,T\}$ for each element $x$, and $\mathcal{K}$ be the collection of such sets.  Then $l1$ and $u1$ approximations of a set $A$ should be given by (i) and (iii) above respectively. Further $A^{l}\,=\,\bigcup\{[x]_{T}\,;\,[x]_{T}\,\subseteq\,A\}\,=\,A^{l1}$ and $A^{u}\,=\,\bigcup\{[x]_{T}\,;\,[x]_{T}\,\cap\,A\,\neq\,\emptyset ,\,x\,\in\,A\}\,\neq\,A^{u1}.$  Note that the generalized cover approach does not prescribe any specific type of granules. 

In \cite{CG98}, the following improved approximations are defined: 
\[A^{l*}\,=\,\{ x\,;\,(\exists{y})\,(x,\,y)\,\in\,T,\,[y]_{T}\,\subseteq\,A\}\]
 \[A^{u*}\,=\,\{x\,;\,(\forall {y})\,((x,\,y)\,\in\,T\,\longrightarrow\,[y]_{T}\,\cap\,A\,\neq\,\emptyset)\} \]

\begin{proposition}
For any subset $A$, $A^{l}\subseteq\,A^{l*}\,\subseteq\,A\,\subseteq\,A^{u*}\,\subseteq\,A^{u}$ 
\end{proposition}

Using this approximation, we can define a Brouwerian orthocomplementation on $\wp(S)$, via $A^{\#}\,=\,\{x\,\in\,S;\,(\forall y\,\in\,A) (x,\,y)\,\notin\,T\}$. These approximations and possible variants may be seen as a way of \emph{distinguishing between objects on a heuristic}. The BZ and Quasi-BZ -algebraic approach (\cite{CC}) make use of a \emph{preclusivity} relation $P$, which is defined via $(a,b)\,\in\,P$ if and only if $(a,\,b)\,\notin\,T$. Semantically the derived rough operators of lower and upper approximation are generated by the preclusivity operator and the complementation relation on the power set of the approximation space or on a collection of sets under suitable constraints in a more abstract setting.  For any $H\,\subseteq\,\underline{S}$, the Brouwerian orthocomplement is $H^{\sharp}\,=\,\{x\,:\,(\forall y \in\,H)\,(x,\,y)\,\in\,P\}$. Then the operators defined by $L_{\sharp}(H)\,=\,H^{c \sharp \sharp c}$ and $M_{\sharp}(H)\,=\,H^{\sharp \sharp}$ behave like lower and upper approximation operators on $\wp(S)$ and are proper generalizations of the corresponding notions in the approximation space context. 

Semantically the BZ-algebra and variants do not capture all the possible ways of arriving at concepts of discernibility over similarity spaces. While the quasi-BZ lattice does not encompass a paradigm shift relative the BZ-algebra, the BZMV variants are designed to capture fuzzy aspects. A major problem with this approach is that the intended semantic domain is uniformly classicalist and at the object level too.

In subjective terms reducts are minimal sets of attributes that preserve the quality of classification. An important problem is in getting good scalable algorithms for the computation of the different types of reducts (or supersets that are close to them) (see ~\cite{KC,BNNSW}). These depend on the concept of granules used. 

Some other notions that we will use are stated below:
\begin{definition}
A double Heyting algebra $L\,=\,\{L,\,\wedge,\,\vee ,\,\rightarrow,\,\ominus,\,0,\,1\}$ is an algebra satisfying:
\begin{itemize}
\item {$L\,=\,\{L,\,\wedge,\,\vee ,\,\rightarrow,\,0,\,1\}$ is a complete atomic Heyting algebra}
\item {$x\,\ominus\,x\,=\,0$; $x\,\vee\,(x\,\ominus\,y)\,=\,x$; $(x\,\ominus\,y)\,\vee,y\,=\,x\,\vee\,y$}
\item {$(x\,\vee\,y)\,\ominus\,z\,=\,(x\,\ominus\,z)\,\vee\,(y\,\ominus\,z)$}
\item {$z\,\ominus\,(x\,\wedge\,y)\,=\,(z\,\ominus\,x)\,\vee\,(z\,\ominus\,y) $}
\end{itemize}
\end{definition}

\begin{definition}
By a \emph{choice function} $\chi$ on a set $S$, we shall mean a function $\chi\,:\wp(S)\,\longmapsto\,S$, which satisfies all of the following:
\begin{itemize}
\item {$(\forall x\,\in\,S)\,\chi (\{x\})\,=\,x$}
\item {$(\forall A\,\in\,\wp(S))\,\chi (A)\,\in\,A$}
\end{itemize}
\end{definition} 

\begin{definition}
Let $P\,=\,\left\langle \underline{P},\,<\,\right\rangle$ be a partially ordered set and if $A$ is any subset of $P$, let its lower and upper cone be $L(A)\,=\,\{x\,;\,(\forall a\,\in\,A)\,x\,\leq\,a\}$ and $U(A)\,=\,\{x\,;\,(\forall a\,\in\,A)\,a\,\leq\,x\}$ respectively. A function $\lambda\,:\,\wp(P)\,\mapsto\,P$ will be said to be \emph{lattice-coherent} with $<$ if and only if 
the condition $a\,\leq\,b$ then $\lambda (L(a,\,b)\,=\,a)$ and $\lambda (U(a,\,b))\,=\,b$.
\end{definition}

$\lambda$-lattices were considered as a generalization of lattices in \cite{SVA}. In the partially ordered set $P$ above  If $\lambda$ is a lattice coherent operation, let $a\,\cdot\,b\,=\,\lambda(L(a,\,b))$ and $a\,+\,b\,=\,\lambda(U(a,\,b))$, then the algebra $Q\,=\,\left\langle \underline{P},\,\cdot,\,+\right\rangle$ is said to be a $\lambda$-\emph{lattice}. It can be shown that if a partially ordered set has a lattice-coherent operation then it can be endowed with a $\lambda$-lattice structure and conversely. 

\begin{theorem}
In a $\lambda$-lattice $Q\,=\,\left\langle \underline{P},\,\cdot,\,+\right\rangle$ all of the following hold:
\begin{itemize}
\item {$+$ is an idempotent and commutative operation }
\item {$\cdot$ is an idempotent and commutative operation}
\item {$a\,\cdot\,(a\,+\,b)\,=\,a$}
\item {$a\,+\,(a\,\cdot\,b)\,=\,a$}
\item {$a\,\cdot\,((a\,\cdot\,b)\,\cdot\,c)\,=\,(a\,\cdot\,b)\,\cdot\,c $}
\item {$a\,+\,((a\,+\,b)\,+\,c)\,=\,(a\,+\,b)\,+\,c $}
\end{itemize} 
\end{theorem}

We will use two kinds of equalities in a partial algebra (of a single sort) $P\,=\,\left\langle \underline{P},\,f_{1},\,\ldots\,f_{n}, \nu \right\rangle $, $f_{i}$s being partial or total function symbols and $\nu$ being an interpretation of these on the set $\underline{P}$. Strictly speaking, the interpreted partial function should be written as $f^{\underline{P}}_{i}$, but we 
will drop the superscript for simplicity. The basic theory of partial algebras can be found in \cite{BU} for example.

For two terms $s$ and $t$, the \emph{strong weak equality} is defined as follows: $t(x)\,\stackrel{\omega^{*}}{=}\,s(x)$ if and only if 
\[(\forall x\,\in\,dom(t)\,=\,dom(s))\,\,t(x)\,=\,s(x).\] 

In contrast, the weak equality is defined via $t(x)\,\stackrel{\omega}{=}\,s(x)$ if and only if 
\[(\forall x\,\in\,dom(t)\cap\,dom(s))\,\,t(x)\,=\,s(x).\] 

\subsection{Bitten Approach}
In this section we recapitulate the essential \emph{bitten} proposal of \cite{SW} with some modifications. Actually the authors introduce \emph{bited} upper approximations in a study on tolerance spaces. However in \cite{SW3}, 'bitten approximations' are also considered. We will uniformly refer to these and approach by the adjective \emph{bitten}. The theory apparently lays emphasis on desired mereological properties at the cost of representation.

Let $Gr(S)\,\subseteq\,\wp(S)$ be the collection of granules for a TAS defined by some conditions including $\bigcup Gr(S)\,=\,S$. A subset $X$ is \emph{granularly definable} if and only if $\exists \mathcal{B}\,\subseteq\,Gr(S)\,X\,=\,\bigcup\mathcal{B}$. The collection of granularly definable sets shall be denoted by $Def_{Gr}(S)$. The lower and upper approximations of $X$ is defined via $Gr_{*}(X)\,=\,\bigcup\{A\,:\,A\,\subseteq\,X,\,A\,\in\,Gr(S) \} $ and 
$Gr^{*}(X)\,=\,\bigcup\{A\,:\,A\,\cap\,X\,\neq\,\emptyset,\,A\,\in\,Gr(S) \}$. The positive and negative region, defined by $POS_{Gr}(X)\,=\,Gr_{*}(X)$ and $NEG_{Gr}(X)\,=\,Gr_{*}(X^{c})$ respectively, are granularly definable. But in this scheme of things $Gr^{*}(X)\,\cap\,NEG_{Gr}(X)\,\neq\,\emptyset$ is possible. To avoid this, a concept of \emph{bitten upper approximation} is defined via $Gr^{*}_{b}(X)\,=\,Gr^{*}(X)\,\setminus\,NEG_{Gr}(X)$. Relative this the boundary is given by \[BN_{Gr}(X)\,=\,Gr^{*}_{b}(X)\,\setminus\,Gr_{*}(X)\,=\,S\,\setminus\,(POS_{Gr}(X)\,\cup\,NEG_{Gr}(X)).\]

$Gr(S)$ may be taken to be the set of $T$-relateds or the set of blocks of $T$ or something else. For example, if $Gr(S)$ is the set of all sets of the form $T_{x}$ ($T_{x}\,=\,\{y\,;\,(x,\,y)\,\in\,T,\,y\,\in\,S\}$), then $S\,=\,\bigcup\,Gr(S)$ and of course $Gr(S)\,\subseteq\,\wp(S)$. Given the latter two properties, the upper and lower approximations of a subset $X$ are then given by : \[Gr_{*}(X)\,=\,\bigcup\,\{Y\,\in\,Gr(S);\,Y\,\subseteq\,X\},\;\;\;Gr^{*}(S)\,=\,\bigcup\,\{Y\,\in\,Gr(S);\,Y\,\cap,X\,\neq\,\emptyset\}.\] The bitten upper approximation is simply $Gr_{b}^{*}(X)\,=\,Gr^{*}(X)\,\setminus\,Gr_{*}(X)$. If they are related to a specific tolerance, then we will call the tuple $\left\langle S,\,Gr(S),\,T,\,Gr_{*},\,Gr_{b}^{*}\right\rangle $ a \emph{bitten approximation system} (\textbf{BAS}). 

The properties of the approximations are as follows:

\begin{center}
\begin{tabularx}{450pt}{|X|X|}
\hline\hline
  \textbf{$l1$-Property}  &  \textbf{$u2$-Property}  \\
\hline
 $1a.) \;\,Gr_{*}(X)\,\subseteq\,X$  &  $1b.) \;\,X\,\subseteq\,Gr_{b}^{*}(X) $\\
\hline
$2a.) \;\,(X\,\subseteq\,Y\,\longrightarrow\,Gr_{*}(X)\,\subseteq\,Gr_{*}(Y))$  &  $2b.) \;\,(X\,\subseteq\,Y\,\longrightarrow\,Gr^{*}_{b}(X)\,\subseteq\,Gr^{*}_{b}(Y)) $\\
\hline
$3a.) \;\,Gr_{*}(\emptyset)\,=\,\emptyset$  &  $3b.) \;\,Gr^{*}_{b}(\emptyset)\,=\,\emptyset $\\
\hline
$4a.) \;\,Gr_{*}(S)\,=\,S$  &  $4b.) \;\,Gr^{*}_{b}(S)\,=\,S $\\
\hline
$5a.) \;\,Gr_{*}(Gr_{*}(X))\,=\,Gr_{*}(X)$  &  $5b.) \;\,Gr^{*}_{b}(Gr^{*}_{b}(X))\,=\,Gr^{*}_{b}(X)$\\
\hline
$6a.) \;\,Gr_{*}(X\,\cap\,Y)\,\subseteq\,Gr_{*}(X)\,\cap\,Gr_{*}(Y)$  &  $6b.) \;\,Gr^{*}_{b}(X\,\cap\,Y)\,\subseteq\,Gr^{*}_{b}(X)\,\cap\,Gr^{*}_{b}(Y)$\\
\hline
$7a.) \;\,Gr_{*}(X)\,\cup\,Gr_{*}(Y)\,\subseteq\,Gr_{*}(X\,\cup\,Y)$  &  $7b.) \;\,Gr^{*}_{b}(X)\,\cup\,Gr^{*}_{b}(Y)\,\subseteq\,Gr^{*}_{b}(X\,\cup\,Y)$\\
\hline
$8a.) \;\,Gr_{*}(X)\,\subseteq\,Gr^{*}_{b}(Gr_{*}(X))$  &  $8b.) \;\,Gr_{*}(Gr_{b}^{*}(X))\,\subseteq) \;\,Gr^{*}_{b}(X) $\\
\hline
$9a.) \;\,(Gr_{*}(X))^{c}\,=\,Gr^{*}_{b}(X^{c})$  &  $9b.) \;\,(Gr_{b}^{*}(X))^{c}\,=\,Gr_{*}(X^{c}) $\\
\hline 
$10A.) \;\,X \in Def_{Gr}(S)\,\longleftrightarrow\,X\,=\,Gr_{*}(X)$  & \\
\hline
$10B.) \;\,X \in Cr_{Gr}(S)\,\longleftrightarrow\,Gr_{*}(X)\,=\,Gr_{b}^{*}(X) $ & \\
\hline
$11A.) \;\,X,\,Y \in Def_{Gr}(S)\,\longrightarrow\,X\,\cup\,Y\,\in\,Def_{Gr}(S)$ & (Implies equality in $7a$) \\
\hline
$11B.) \;\,X,\,Y \in Cr_{Gr}(S)\,\longrightarrow\,X\,\cap\,Y,\,X\,\cup\,Y\,\in\,Cr_{Gr}(S)$ & (Implies equality in $6a,\,6b,\,7a,\,7b,\,8a,\,8b$) \\
\lasthline
\end{tabularx}
\end{center} 

In the same paper (\cite{SW}), the authors pose the problem of developing an algebraic semantics for the approach.

In a TAS $S$, a \emph{block} $B\,\subseteq\, S$ is a maximal set that satisfies $B^{2}\,\subseteq\,T$. In \cite{SW3}, these are termed \emph{classes}, while the former is the standard terminology in universal algebra. If $\mathcal{H}_{T}$ is the collection of all blocks of $T$, then let $\mathcal{A}_{T}\,=\,\{\cap F\,:\, F\,\subseteq\,\mathcal{H}_{T}\}$. Taking this large collection as the set of granules, the authors define the the lower and bitten upper approximation of a $X\,\subseteq\, S$ as $X^{l}\,=\,\bigcup\{A:\,A\,\subseteq\,X,\,A\,\in\,\mathcal{A}_{T}\}$ and  $X^{u}_{b}\,=\,\bigcup\{A:\,A\,\cap\,X\,\neq\,\emptyset ,\, A\,\in\,\mathcal{A}_{T}\}\,\setminus (X^{c})^{l}$. On the set of definable objects $\Delta (S)$, let 
\begin{itemize}
\item{$X\,\rightarrow\,Y\,=\,\bigcup \{A\,\in\,\mathcal{A}_{T};\,X\,\cap\,A\,\subseteq Y\}$} 
\item {$X\,\ominus Y\,=\,\bigcap \{B\,\in\,\mathcal{A}_{T};\,X\,\subseteq\,Y\,\cup\,A\}$.}
\end{itemize}
Then the following theorem provides a topological algebraic semantics (\cite{SW3}):
\begin{theorem}\label{gra}
$\left\langle \Delta (S),\,\cap,\,\cup,\,\rightarrow,\,\ominus,\,\emptyset,\,S\right\rangle $ is a complete atomic double Heyting algebra. 
\end{theorem}

However extensions of the theorem to other types of granules are not known. No abstract representation theorem is also proved in the particular situation. 

\section{Semantics for Bitten Rough Set Theory}

The concept of granules to be used in the theory is essentially kept open. Many types of granules may not permit nice represention theory. Despite this, our first semantics over roughly equivalent objects does well. This is due to the higher order approach used.

\begin{definition}
If $S$ is a TAS, then over $\wp (S)$ let \[A\,\sim\,B\;\mathrm{iff}\;Gr_{*}(A)\,=\,Gr_{*}(B)\;\mathrm{and}\;Gr^{*}_{b}(A)\,=\,Gr^{*}_{b}(B) \]  
\end{definition}

The following proposition and theorem basically say \emph{the quotient structure (or the set of roughly equivalent objects) has very little structure} with respect to desirable properties of a partial algebra. They are clearly deficient from the rough perspective as we do not have proper conjunction and disjunction operations.  
But 'biting' may actually make the partial operations total in many contexts. 

\begin{proposition}\label{sec3}
$\sim$ is an equivalence on the power set $\wp (S)$. Moreover the following operations and relations on $\wp (S)|\sim$ are well-defined:
\begin{itemize}
\item {$L([A])\,=\,[Gr_{*}(A)]$}
\item {$\neg[A]\,=\,[A^{c}]$ if defined}
\item {$\blacklozenge ([A])\,=\,[Gr^{*}_{b}(A)]$}
\item {$[A]\,\leq\,[B]$ if and only if, for any $A\,\in [A]$ and $B\,\in\,[B]$ $Gr_{*}(A)\,\subseteq\,Gr_{*}(B)$ and $Gr^{*}_{b}(A)\,\subseteq\,Gr^{*}_{b}(B)$.}
\item {$[A]\,\Cap\,[B]\,=\,[C]$ if and only if $[C]$ is the infimum of $[A]$ and $[B]$ w.r.t $\leq$. It shall be taken to be undefined in other cases. }
\item {$[A]\,\Cup\,[B]\,=\,[C]$ if and only if $[C]$ is the supremum of $[A]$ and $[B]$ w.r.t $\leq$. It shall be taken to be undefined in other cases. }
\end{itemize}

Moreover $\leq$ is a partial order on $\wp (S)|\sim$ that is partially compatible with $L$ on the crisp elements. 
\end{proposition}

\begin{proof} 
\begin{itemize}
\item {For any $A\,\in\,\wp(S)$ and any $B\,\in\,[A]$, $Gr_{*}(Gr_{*}(A)\,=\,Gr_{*}(A)\,=\,Gr_{*}(B)$ and $Gr^{*}_{b}(Gr_{*}(A))\,=\,Gr^{*}_{b}(Gr_{*}(B))$. This proves that $L$ is  well-defined. }

\item {For any $A\,\in\,\wp(S)$ and any $B,\,E\in\,[A]$, $(Gr_{*}(B))^{c}\,=\,Gr_{b}^{*}(B^{c})\,=\,(Gr_{*}(E))^{c}\,=\,Gr_{b}^{*}(E^{c})$ and $(Gr_{b}^{*}(B))^{c}\,=\,Gr_{*}(B^{c)}\,=\,(Gr_{b}^{*}(E))^{c}\,=\,Gr_{*}(E^{c)}$.
}
\end{itemize}
We can verify the rest in a similar way.  
\end{proof}

\begin{proposition}
In the above context, \[([A]\,\leq\,[B]\,\longrightarrow\,\neg [B]\,\leq\,\neg [A])\] 
\end{proposition}

\begin{theorem}
All of the following hold in $\wp (S)|\sim$ (we assume that unary operators bind more strongly than binary ones):
\begin{enumerate}
\item {$(\blacklozenge x\,\Cup\,\blacklozenge y\,=\,a,\,\blacklozenge (x\,\Cup\,y)\,=\,b\,\longrightarrow\,\blacklozenge x\,\Cup\,\blacklozenge y\,\leq \,\blacklozenge (x\,\Cup\,y))$}
\item {$x\,\Cap\,\blacklozenge x\,=\,x$}
\item {$\blacklozenge \blacklozenge x\,=\,\blacklozenge x$}
\item {$(\blacklozenge x\,\Cap\,\blacklozenge y\,=\,a,\,\blacklozenge (x\,\Cap\,y)\,=\,b\,\longrightarrow\,\blacklozenge (x\,\Cap\,y)\,\leq\,\blacklozenge x\,\Cap\,\blacklozenge y) $}
\item {$(L x\,\Cup\,L y\,=\,a,\,L (x\,\Cup\,y)\,=\,b\,\longrightarrow\,L x\,\Cup\,L y\,\leq\,L (x\,\Cup\,y))$}
\item {$L x\,\Cap\,\blacklozenge L x\,=\,L x$}
\item {$\blacklozenge x\,\Cup\,L \blacklozenge  x\,=\,\blacklozenge x$}
\item {$\neg \blacklozenge x\,=\,L \neg x$}
\item {$\neg L x\,=\,\blacklozenge \neg x$}
\end{enumerate}
\end{theorem}

\begin{proof}
\begin{enumerate}
\item {If $A\,\in\,x$ and $B\,\in\,y$, then $\blacklozenge x\,=\,[Gr^{*}_{b}(A)]$, $\blacklozenge y\,=\,[Gr^{*}_{b}(B)]$. Given the existence of the terms in the premise, we can assume that there exists $C\,\in\,[Gr^{*}_{b}(A)]\,\Cup\,[Gr^{*}_{b}(B)]$ and $E\,\in\,\blacklozenge (x\,\Cup\,y)$. $Gr_{*}(C)\,\subseteq\,Gr_{*}(E)$ and $Gr^{*}_{b}(C)\,\in\,Gr^{*}_{b}(E)$. So, given the existence of the terms in the premise, we have $\blacklozenge x\,\Cup\,\blacklozenge y\,\leq \,\blacklozenge (x\,\Cup\,y)$.}
\item {If $A\,\in\,x$, then $\blacklozenge x\,=\,[Gr^{*}_{b}(A)]$. As $A\,\subseteq\,Gr^{*}_{b}(A)$, so $x\,\Cap\,\blacklozenge x\,=\,x$.}
\item {If $A\,\in\,x$, then $\blacklozenge \blacklozenge x\,=\,\blacklozenge [Gr^{*}_{b}(A)]\,=\,[Gr^{*}_{b}(Gr^{*}_{b}(A))]\,=\,[Gr^{*}_{b}(A)] \,=\,\blacklozenge x$.}
\item {The proof of $(\blacklozenge x\,\Cap\,\blacklozenge y\,=\,a,\,\blacklozenge (x\,\Cap\,y)\,\longrightarrow\,\blacklozenge (x\,\Cap\,y)\,\leq\,\blacklozenge x\,\Cap\,\blacklozenge y)$ is similar to that of the first item.}
\item {Given the existence of the terms in the premise, if $A\,\in\,x$ and $B\,\in\,y$, then $Lx\,=\,L[A]\,=\,[Gr_{*}(A)]$ and $Ly\,=\,[Gr_{*}(B)]$. If $C\,\in\,Lx\,\Cup\,Ly$ and $E\,\in\,L(x\,\Cup\,y)$, then $Gr_{*}(C)\,\subseteq\,Gr_{*}(A)\,\cup\,Gr_{*}(B)$, $Gr_{*}(A)\,\cup\,Gr_{B}\,\subseteq\,Gr_{*}(E)$ and $Gr^{*}_{b}(C)\,\subseteq\,Gr^{*}_{b}(E)$. So, $L x\,\Cup\,L y\,\leq\,L (x\,\Cup\,y)$.}
\item {If $A\,\in\,x$, then $Lx\,=\,L[A]\,=\,[Gr_{*}(A)]$ and $\blacklozenge L x\,=\,[Gr^{*}_{b} Gr_{*}(A)]$. But  $Gr_{*}(A)\,\subseteq\,Gr^{*}_{b}Gr_{*}(A)$. So $L x\,\Cap\,\blacklozenge L x\,=\,L x$.}
\item {If $A\,\in\,x$, then $\blacklozenge x\,=\,\blacklozenge [A]\,=\,[Gr^{*}_{b}(A)]$ and $L \blacklozenge x\,=\,[Gr_{*} Gr^{*}_{b}(A)]$. But  $Gr_{*} Gr^{*}_{b}(A)\,\subseteq\,Gr^{*}_{b}(A)$. So, $\blacklozenge x\,\Cup\,L \blacklozenge  x\,=\,\blacklozenge x$.}
\item {If $A\,\in\,x$, then $\neg \blacklozenge x\,=\,\neg \blacklozenge [A]\,=\,\neg [Gr^{*}_{b}(A)]$. $\neg [Gr^{*}_{b}(A)]\,=\,[(Gr^{*}_{b}(A))^{c}],=\,[Gr_{*}(A^{c})]\,=\,L neg [A]$. So, $\neg \blacklozenge x\,=\,L \neg x$.}
\item {If $A\,\in\,x$, then $\neg L x\,=\,\neg [Gr_{*}(A)]\,=\,[(Gr_{*}(A))^{c}]$. But $[(Gr_{*}(A))^{c}]\,=\,[Gr^{*}_{b}(A^{c})]$. This yields $\neg L x\,=\,\blacklozenge \neg x$.}
\end{enumerate} 
\end{proof}

A semantics using the partial algebra over the associated quotient may be difficult because of axiomatisability issues. So we use a higher order approach, taking care not to introduce extraneous properties. Eventually the constructed algebra ends up with three partial orders. In the following construction the use of a modified concept of filters simplifies the eventual representation theorem.

If $\wp (S)|\sim\,=\,K$, then let $K^{*}\,=\,\{f\,:\,f\,:\,K\,\mapsto I\,\,\mathrm{is}\,\,\mathrm{isotone} \}$, $I$ being the the totally ordered two element set $\{0,\,1\}$ under $0\,<\,1$. For any $A\,\subset\,K^{*}$, a subset $F$ is an $A$-\emph{ideal} if and only if \[F\,=\,\bigcap_{x\,\in\,A}\,x^{-1}\{0\}.\] Dually  
$F$ is an $A$-\emph{filter} if and only if \[F\,=\,\bigcap_{x\,\in\,A}\,x^{-1}\{1\}.\]

All $A$-ideals are order ideals (w.r.t the induced order on $K^{*}$), but the converse need not hold. $A\,\subset\,K^{*}$ is said to be \emph{full} if $\forall \,p\,\nleq\,q\,\exists\, x\,\in\,A\, x(p)\,=\,1,\,x(q)\,=\,0$. $A$ is said to be \emph{separating} if for any disjoint ideal $I$ and filter $F$, there exists a $x\,\in\,A$ such that $x_{|I}\,=\,0$ and $x_{|F}\,=\,1$ 

\begin{lemma}
If $A$ is a separating subset of $K^{*}$ and $(\forall p,\,q\,\in\,K)(q\,\nleq\,p\,\longrightarrow\,p \downarrow_{A}\,\cap\,q \uparrow_{A}\,=\,\emptyset)$, then $A$ is full.  
\end{lemma}

If $p\,\in\,K$, then let $\mathcal{UP}(p)\,=\,\{x\,:\,x(p)\,=\,1\}$ and $\mathcal{LO}(p)\,=\,\{x\,:\,x(p)\,=\,0\}$, then we can define two closure operators $C_{1},\,C_{2}$ via \[C_{1}\,=\,clos\{\mathcal{UP}(p)\}_{p\,\in\,K} \]
(a $C_{1}$-closed set is an intersection of elements of $\{\mathcal{UP}(p)\}_{p\,\in\,K}$) and \[C_{2}\,=\,clos\{\mathcal{LO}(p)\}_{p\,\in\,K}\]

Note that elements of $\mathcal{UP}(p)\}_{p\,\in\,K}$ are in fact $C_{1}O_{2}$-sets (that is sets that are open w.r.t the second closure system and closed with respect to the first). The set of such sets on a system $(S,\,C_{1},\,C_{2})$, will be denoted by $C_{1}O_{2}(S,\,C_{1},\,C_{2})$. 
The associated closure operators will be denoted by $cl_{1}$ and $cl_{2}$ respectively.

On any subset $A\,\subseteq\,K^{*}$, we can define closure operators via $C_{iA}(X)\,=\,C_{i}\,\cap\,A$, with associated closure systems $\mathcal{UP}_{A}(p)\,=\,\mathcal{UP}(p)\,\cap\,A$ and $\mathcal{LO}_{A}(p)\,=\,\mathcal{LO}(p)\,\cap\,A$ respectively. It can be seen that, in the situation, $C_{1A}\,=\,clos \{\mathcal{UP}(p)\}_p\,\in\,P$ and $C_{2A}\,=\,clos \{\mathcal{LO}(p)\}_p\,\in\,P$.  

\begin{theorem}
If $A\,\subseteq\,K^{*}$ and $\sigma\,:\,K\,\mapsto\,C_{1}O_{2}(A,\,C_{1A},\,C_{2A})$ is a map defined by $\sigma (p)\,=\,\mathcal{UP}(p)$ then 
\begin{enumerate}
\item {$\sigma$ is isotone}
\item {If $A$ is full, then $\sigma$ is injective }
\item {If $A$ is separating, then $\sigma$ is surjective.}
\end{enumerate}
\end{theorem}

\begin{proof}
This theorem and the following theorem are proved for an arbitrary partially ordered set $K$ in \cite{BZ}.    
\end{proof}

\begin{theorem}
If $A$ is a full and separating subset of $K^{*}$, then $K\,\cong\,C_{1}O_{2}(A,\,C_{1A},\,C_{2A})$. In particular, $K\,\cong\,C_{1}O_{2}(K^{*},\,C_{1},\,C_{2})$ (as $K^{*}$ is a full and separating set). Even $K^{*}\setminus\{0,1\}$ is a full and separating set.
\end{theorem}

We state the following for clarifying the connection with the more common way of using closure operators.
\begin{proposition}
$A\,\in\,C_{1}O_{2}(K^{*},\,C_{1},\,C_{2})$ if and only if $A\,\subseteq\,K^{*}$ and $(\exists B,\,E\,\subseteq\,K^{*})\,A\,=\,cl_{1}B,\,A\,=\,K^{*}\,\setminus\,cl_{2}(E)$.   
\end{proposition}
  
$K^{*}$ can be interpreted as the set of partitions of the set of roughly equivalent elements into an upper and lower region subject to the new order being a coarsening of the original order. The important thing is that this restricted global object is compatible with the 'natural global' versions of the other operations and leads to a proper semantics. We show this in what follows.

\begin{definition}
On $K^{*}$, the following global operations (relative those on $K$) can be defined:
\begin{itemize}
\item {If $A\,\in\,C_{1}O_{2}(K^{*},\,C_{1},\,C_{2})$, then $\mathfrak{L}(A)\,=\,L(i(A))$, $i$ being the canonical identity map from $C_{1}O_{2}(K^{*},\,C_{1},\,C_{2})$ onto $\wp{K}$. }
\item {$A\,\vee\,B\,=\,cl_{1}(A\,\cup\,B)$ (if the RHS is also open with respect to the second closure system), $\cup$ being the union operation over $K^{*}$ }
\item {$A\,\wedge\,B\,=\,cl_{1}(A\,\cap\,B)$ (if the RHS is also open with respect to the second closure system), $\cap$ being the intersection operation over $K^{*}$ }
\item {If $A\,\in\,C_{1}O_{2}(K^{*},\,C_{1},\,C_{2})$ then $\diamondsuit A\,=\,\blacklozenge i(A)$ }
\item {If $A\,\in\,C_{1}O_{2}(K^{*},\,C_{1},\,C_{2})$ then $\sim A\,=\,\neg i(A)$}
\item {$cl_{1},\,cl_{2}$ can be taken as unary operators on $K^{*}$ }
\item {$1,\,\bot,\,\top,$ shall be $0$-ary operations with interpreted values corresponding to $K,\,\emptyset,\,K^{*}$ respectively } 
\item {If $A,\,B\,\in\,C_{1}O_{2}(K^{*},\,C_{1},\,C_{2})$ then $A\,\sqcap\,B,=\,i(A)\,\Cap\,i(B)$ }
\item {If $A,\,B\,\in\,C_{1}O_{2}(K^{*},\,C_{1},\,C_{2})$ then $A\,\sqcup\,B,=\,i(A)\,\Cup\,i(B)$}
\end{itemize}
\end{definition}
\begin{proposition}
A partial operation is well defined if it is either uniquely defined or not ambiguously defined. In this sense all of the operations and partial operations are well-defined. 
\end{proposition}
\begin{proof}
Most of the verification is direct. 
\end{proof}

\begin{definition}
An algebra of the form  \[\mathfrak{W}\,=\,\left\langle\underline{\wp(K^{*})},\,\vee,\,\wedge,\,\sqcap,\,\sqcup,\,\cup,\,\cap,\,^{c},\,,cl_{1},\,cl_{2},\,\sim,\,\mathfrak{L},\,\diamondsuit,\,\bot,\,1,\,\top \right\rangle \] of type $(2,\,2,\,2,\,2,\,1,\,1,\,1,\,1,\,1,\,1,\,0,\,0,\,0)$ in which the operations are as in the above definition will be called a \emph{concrete bitten algebra}.  $\cup,\,\cap,\,^{c}$ are the union, intersection and complementation operations respectively on $\wp (K^{*})$. We use $\xi (x)$ as an abbreviation for $cl_{1}x\,=\,x,\,cl_{2} x^{c}\,=\,x^{c}$. Further $\xi (x,\,y,\,\ldots )$ shall mean $\xi (x)$, $\xi (y)$ and so on. If $S$ is the original TAS, then we will denote its associated concrete bitten algebra by $Bite(S)$.
\end{definition}
  
\begin{theorem}
A concrete Bitten Algebra $\mathfrak{W}$ satisfies all of the following:

\begin{enumerate}
\item {$\left\langle\underline{\wp(K^{*})},\,\cup,\,\cap,\,^{c},\,\bot,\,\top \right\rangle $ is a Boolean algebra. Note that after forming associated categories with the usual concept of morphisms, we can realize this through forgetful functors.}
\item {$x\,\vee\,y\,\stackrel{\omega^{*}}{=}\,y\,\vee\,x$}
\item {$x\,\vee\,(y\,\vee\,z)\stackrel{\omega}{=}\,(x\,\vee\,y)\,\vee\,z$}
\item {$(\xi (x)\,\longrightarrow\,x\,\vee\,x\,=\,cl_{1}(x))$ }
\item {$(x\,\vee\,y\,=\,z\,\longrightarrow\,cl_{1}z\,=\,z,\,cl_{2} z^{c}\,=\,z^{c})$}
\item {$(x\,\vee\,x\,=\,y\,\longrightarrow\,cl_{2}(x^{c})\,=\,x^{c},\,y\,=\,cl_{1}(x)\,=\,x\,\wedge\,x)$}
\item {$cl_{i}(x)\,\cap\,x\,=\,x $;  $cl_{i} cl_{i} (x)\,=\,cl_{i} (x)$;  $i\,=\,1,\,2$}
\item {$(x\,\cap\,y\,=\,x\,\longrightarrow\,cl_{i}(x)\,\cap\,cl_{i}(y)\,=\,cl_{i}(x))$;  $i\,=\,1,\,2$}
\item {$x\,\wedge\,y\,\stackrel{\omega^{*}}{=}\,y\,\wedge\,x$}
\item {$x\,\wedge\,(y\,\wedge\,z)\stackrel{\omega}{=}\,(x\,\wedge\,y)\,\vee\,z$}
\item {$(x\,\wedge\,x\,=\,y\,\longrightarrow\,cl_{2}(x^{c})\,=\,x^{c},\,y\,=\,cl_{1}(x))$}
\item {$(cl_{2}((cl_{1}x) ^{c})\,=\,(cl_{1}x)^{c}\,\longrightarrow\,x\,\wedge\,x\,=\,cl_{1}(x))$ }
\item {$(x\,\wedge\,y\,=\,z\,\longrightarrow\,cl_{1}z\,=\,z,\,cl_{2} z^{c}\,=\,z^{c})$}
\item {$((x\,\wedge\,y)\,\vee\,x\,=\,z\,\longrightarrow\,z\,=\,cl_{1}(x))$}
\item {$(x\,\wedge\,y)\,\vee\,x\,\stackrel{\omega^{*}}{=}\,x\,\wedge\,(y\,\vee\,x) $}
\item {$\ml \bot\,=\,\bot$;  $\ml 1\,=\,1$}
\item {$(cl_{1}x\,=\,x,\,cl_{2}(x^{c})\,=\,x^{c}\,\longrightarrow\,x\,\vee\,\ml x\,=\,x,\,\ml \ml x\,=\,\ml x)$}
\item {$(x\,\sqcap\,y\,=\,z\,\longrightarrow\,\xi(x,\,y,\,z))$}
\item {$ (x\,\sqcup\,y\,=\,z\,\longrightarrow\,\xi(x,\,y,\,z))$}
\item {$x\,\sqcap\,y\,\stackrel{\omega^{*}}{=}\,y\,\sqcap\,x$}
\item {$x\,\sqcap\,(y\,\sqcap\,z)\stackrel{\omega}{=}\,(x\,\sqcap\,y)\,\sqcap\,z$}
\item {$x\,\sqcup\,y\,\stackrel{\omega^{*}}{=}\,y\,\sqcup\,x$}
\item {$x\,\sqcup\,(y\,\sqcup\,z)\stackrel{\omega}{=}\,(x\,\sqcup\,y)\,\sqcup\,z$}
\item {$(\xi(x,\,y)\,\longrightarrow\,\diamondsuit (x\,\sqcup\,y)\,\cap\,(\diamondsuit x\,\sqcup\,\diamondsuit y)\,=\,\diamondsuit x\,\sqcup\,\diamondsuit y)$}
\item {$(\xi(x)\,\longrightarrow\,x\,\sqcap\,\diamondsuit x\,=\,x,\,\,\diamondsuit \diamondsuit x\,=\,\diamondsuit x)$}
\item {$(\xi(x,\,y)\,\longrightarrow\,\diamondsuit (x\,\sqcap\,y)\,\cap\,(\diamondsuit x\,\sqcap\,\diamondsuit y)\,=\,\diamondsuit (x\,\sqcup\, y))$}
\item {$(\xi(x,\,y)\,\longrightarrow\,\ml (x\,\sqcup\,y)\,\cap\,(\ml x\,\sqcup\,\ml y)\,=\,\ml x\,\sqcup\,\ml y)$}
\item {$(\xi (x)\,\longrightarrow\,\ml x\,\sqcap\,\diamondsuit \ml x\,=\,\ml x)$}
\item {$(\xi (x)\,\longrightarrow\,\diamondsuit x\,\sqcup\,\ml \diamondsuit  x\,=\,\diamondsuit x)$}
\item {$(\xi x\,\longrightarrow\,\sim \diamondsuit x\,=\,\ml \sim x)$}
\item {$(\xi x\,\longrightarrow\,\sim \ml x\,=\,\diamondsuit \sim x)$}
\item {$(x\,\sqcap\,y\,=\,x\,\longrightarrow\,x\,\sqcup\,y\,=\,y)$}
\end{enumerate}
\end{theorem}

\begin{proof}
\begin{enumerate}
\item {That $\left\langle\underline{\wp(K^{*})},\,\cup,\,\cap,\,^{c},\,\bot,\, \top \right\rangle$ is a Boolean algebra can be proved by part of Stone's representation theorem.}
\item {For proving $x\,\vee\,y\,\stackrel{\omega^{*}}{=}\,y\,\vee\,x$, if $x\,\vee\,y$ is defined, then $cl_{1}(x\,\cup\,y)$ is open with respect to $cl_{2}$. So $cl_{1}(y\,\cup\,x$ is also open with respect to $cl_{2}$ and the two sides of the equality must be equal. Similarly for the reversed argument.}
\item {If $x\,\vee\,(y\,\vee\,z)$ and $(x\,\vee\,y)\,\vee\,z$ are defined, then they must equal $cl_{1}(x\,\cup\,cl_{1}(y\,\cup\,z))$ and $cl_{1}(cl_{1}(x\,\cup\,y))\,\cup\,z)$ respectively. Further these and $cl_{1}(x\,\cup\,y)$, and $cl_{1}(y\,\cup\,z)$ must be open with respect to $cl_{2}$. But $cl_{i}$ are topological closures. So $x\,\vee\,(y\,\vee\,z)\stackrel{\omega}{=}\,(x\,\vee\,y)\,\vee\,z$}
\item {$(\xi (x)\,\longrightarrow\,x\,\vee\,x\,=\,cl_{1}(x))$ can be derived directly.}
\item {If $(x\,\vee\,y\,=\,z$ then $z\,=\,cl_{1}(x\,\cup\,y)$ and it must be open with respect to the second closure system. So, 
$(x\,\vee\,y\,=\,z\,\longrightarrow\,cl_{1}z\,=\,z,\,cl_{2} z^{c}\,=\,z^{c})$}
\item {If $x\,\vee\,x\,=\,y$, then $x\,\wedge\,x$ will also be equal to $cl_{1}(x)$. The rest of $(x\,\vee\,x\,=\,y\,\longrightarrow\,cl_{2}(x^{c})\,=\,x^{c},\,y\,=\,cl_{1}(x)\,=\,x\,\wedge\,x)$ follows from the previous observation.}
\item {$cl_{i}(x)\,\cap\,x\,=\,x $;  $cl_{i} cl_{i} (x)\,=\,cl_{i} (x)$;  $i\,=\,1,\,2$ follows from definition}
\item {$(x\,\cap\,y\,=\,x\,\longrightarrow\,cl_{i}(x)\,\cap\,cl_{i}(y)\,=\,cl_{i}(x))$;  $i\,=\,1,\,2$ expresses monotonicity}
\item {The proof of $x\,\wedge\,y\,\stackrel{\omega^{*}}{=}\,y\,\wedge\,x$ is similar to that of its dual.}
\item {The proof of $x\,\wedge\,(y\,\wedge\,z)\stackrel{\omega}{=}\,(x\,\wedge\,y)\,\vee\,z$ is similar but easier than that of its dual.}
\item {$(x\,\wedge\,x\,=\,y\,\longrightarrow\,cl_{2}(x^{c})\,=\,x^{c},\,y\,=\,cl_{1}(x))$ follows directly from definition.}
\item {$(cl_{2}(cl_{1}x)^{c})\,=\,(cl_{1}x)^{c}\,\longrightarrow\,x\,\wedge\,x\,=\,cl_{1}(x))$ is also direct.}
\item {If $x\,\wedge\,y\,=\,z$, then $z$ must necessarily be closed with respect to $cl_{1}$ and open with respect to $cl_{2}$. So $(x\,\wedge\,y\,=\,z\,\longrightarrow\,cl_{1}z\,=\,z,\,cl_{2} z^{c}\,=\,z^{c})$}
\item {$(x\,\wedge\,y)\,=\,a$ (say) is certainly lesser than $cl_{1}(x)$, so  $(x\,\wedge\,y)\,\vee\,x\,=\,cl_{1}(a\,\cup\,x)\,=\,cl_{1}(x)$. This proves $((x\,\wedge\,y)\,\vee\,x\,=\,z\,\longrightarrow\,z\,=\,cl_{1}(x))$}
\item {The argument of the previous conditional implication can be extended to prove $(x\,\wedge\,y)\,\vee\,x\,\stackrel{\omega^{*}}{=}\,x\,\wedge\,(y\,\vee\,x) $.}
\item {$\ml \bot\,=\,\bot$, and  $\ml 1\,=\,1$ follow from definition.}
\item {If $(cl_{1}x\,=\,x,\,cl_{2}(x^{c})\,=\,x^{c}$, then $x$ is essentially in the main quotient structure of interest. So $\ml x$ will be defined and the rest of $(cl_{1}x\,=\,x,\,cl_{2}(x^{c})\,=\,x^{c}\,\longrightarrow\,x\,\vee\,\ml x\,=\,x,\,\ml \ml x\,=\,\ml x)$ follows.}
\item {If $(x\,\sqcap\,y\,=\,z$, then $x,\,y$ are essentially in $C_{1}O_{2}(K^{*},\,C_{1},\,C_{2}$, but then $x\,\sqcap\,y$ must be the infimum of $x$ and $y$ with respect to $\leq$. So $z$ must also be in $C_{1}O_{2}(K^{*},\,C_{1},\,C_{2}$ and $(x\,\sqcap\,y\,=\,z\,\longrightarrow\,\xi(x,\,y,\,z))$.}
\item {The proof of $ (x\,\sqcup\,y\,=\,z\,\longrightarrow\,\xi(x,\,y,\,z))$ is similar to that of the above statement.}
\item {If either $x\,\sqcap\,y$ or $y\,\sqcap\,x$ is defined, then the other is and the two must be equal to their identification with $i(x)\,\Cap\,i(y)$. So $x\,\sqcap\,y\,\stackrel{\omega^{*}}{=}\,y\,\sqcap\,x$.}
\item {$x\,\sqcap\,(y\,\sqcap\,z)\stackrel{\omega}{=}\,(x\,\sqcap\,y)\,\sqcap\,z$}
\item {$x\,\sqcup\,y\,\stackrel{\omega^{*}}{=}\,y\,\sqcup\,x$ can be proved in the same way as its dual statement (with $\sqcap$).}
\item {The rest of the proof follows from the first theorem of this section.}
\end{enumerate}
\end{proof}

We have defined a concrete bitten algebra in such a way that abstraction becomes easy.

\begin{definition}
By an \emph{abstract bitten algebra}, we shall mean a partial algebra of the form \\ $A\,=\,\left\langle\underline{A},\,\vee,\,\wedge,\,\sqcap,\,\sqcup,\,\cup,\,\cap,\,^{c},\,,cl_{1},\,cl_{2},\,\sim,\,\mathfrak{L},\,\diamondsuit,\,\bot,\,1,\,\top \right\rangle $ that satisfies all of the following:
\begin{enumerate}
\item {$\left\langle\underline{A},\,\cup,\,\cap,\,^{c},\,\bot,\,\top \right\rangle $ is a Boolean algebra.}
\item {$x\,\vee\,y\,\stackrel{\omega}{=}\,y\,\vee\,x$}
\item {$x\,\vee\,(y\,\vee\,z)\stackrel{\omega}{=}\,(x\,\vee\,y)\,\vee\,z$}
\item {$(cl_{2} (cl_{1}x)^{c})\,=\,(cl_{1}x)^{c}\,\longrightarrow\,x\,\vee\,x\,=\,cl_{1}(x))$ }
\item {$(x\,\vee\,y\,=\,z\,\longrightarrow\,cl_{1}z\,=\,z,\,cl_{2} z^{c}\,=\,z^{c})$}
\item {$(x\,\vee\,x\,=\,y\,\longrightarrow\,cl_{2}(x^{c})\,=\,x^{c},\,y\,=\,cl_{1}(x)\,=\,x\,\wedge\,x)$}
\item {$cl_{i}(x)\,\cap\,x\,=\,x $;  $cl_{i} cl_{i} (x)\,=\,cl_{i} (x)$;  $i\,=\,1,\,2$}
\item {$(x\,\cap\,y\,=\,x\,\longrightarrow\,cl_{i}(x)\,\cap\,cl_{i}(y)\,=\,cl_{i}(x))$;  $i\,=\,1,\,2$}
\item {$x\,\wedge\,y\,\stackrel{\omega}{=}\,y\,\wedge\,x$}
\item {$x\,\wedge\,(y\,\wedge\,z)\stackrel{\omega}{=}\,(x\,\wedge\,y)\,\vee\,z$}
\item {$(x\,\wedge\,x\,=\,y\,\longrightarrow\,cl_{2}(x^{c})\,=\,x^{c},\,y\,=\,cl_{1}(x))$}
\item {$(cl_{2}((cl_{1}x) ^{c})\,=\,(cl_{1}x)^{c}\,\longrightarrow\,x\,\wedge\,x\,=\,cl_{1}(x))$ }
\item {$(x\,\wedge\,y\,=\,z\,\longrightarrow\,cl_{1}z\,=\,z,\,cl_{2} z^{c}\,=\,z^{c})$}
\item {$((x\,\wedge\,y)\,\vee\,x\,=\,z\,\longrightarrow\,z\,=\,cl_{1}(x))$}
\item {$(x\,\wedge\,y)\,\vee\,x\,\stackrel{\omega}{=}\,x\,\wedge\,(y\,\vee\,x) $}
\item {$\ml \bot\,=\,\bot$;  $\ml 1\,=\,1$}
\item {$(cl_{1}x\,=\,x,\,cl_{2}(x^{c})\,=\,x^{c}\,\longrightarrow\,x\,\vee\,\ml x\,=\,x,\,\ml \ml x\,=\,\ml x)$}
\item {$(x\,\sqcap\,y\,=\,z\,\longrightarrow\,\xi(x,\,y,\,z))$}
\item {$ (x\,\sqcup\,y\,=\,z\,\longrightarrow\,\xi(x,\,y,\,z))$}
\item {$x\,\sqcap\,y\,\stackrel{\omega}{=}\,y\,\sqcap\,x$}
\item {$x\,\sqcap\,(y\,\sqcap\,z)\stackrel{\omega}{=}\,(x\,\sqcap\,y)\,\sqcap\,z$}
\item {$x\,\sqcup\,y\,\stackrel{\omega}{=}\,y\,\sqcup\,x$}
\item {$x\,\sqcup\,(y\,\sqcup\,z)\stackrel{\omega}{=}\,(x\,\sqcup\,y)\,\sqcup\,z$}
\item {$(\xi(x,\,y)\,\longrightarrow\,\diamondsuit (x\,\sqcup\,y)\,\cap\,(\diamondsuit x\,\sqcup\,\diamondsuit y)\,=\,\diamondsuit x\,\sqcup\,\diamondsuit y)$}
\item {$(\xi(x,\,y)\,\longrightarrow\,\diamondsuit (x\,\sqcap\,y)\,\cap\,(\diamondsuit x\,\sqcap\,\diamondsuit y)\,=\,\diamondsuit (x\,\sqcup\, y))$}
\item {$(\xi(x,\,y)\,\longrightarrow\,\ml (x\,\sqcup\,y)\,\cap\,(\ml x\,\sqcup\,\ml y)\,=\,\ml x\,\sqcup\,\ml y)$}
\item {$(\xi (x)\,\longrightarrow\,\ml x\,\sqcap\,\diamondsuit \ml x\,=\,\ml x)$}
\item {$(\xi (x)\,\longrightarrow\,\diamondsuit x\,\sqcup\,\ml \diamondsuit  x\,=\,\diamondsuit x)$}
\item {$(\xi x\,\longrightarrow\,\sim \diamondsuit x\,=\,\ml \sim x)$}
\item {$(\xi x\,\longrightarrow\,\sim \ml x\,=\,\diamondsuit \sim x)$}
\end{enumerate}
\end{definition}

\begin{definition}
Let $\tau$ be a collection of subsets of $K$ indexed by $K$, that satisfies
\begin{enumerate}
\item {$(\forall x \in K)(\exists \!y\,\in\,\tau)\,x\,\in\,y$}
\item {$\bigcup\,\tau\,=\,K$}
\item {$\tau$ is an antichain with respect to inclusion}
\item {For a not necessarily disjoint partition $\mathcal{P}$ of $K$,  $\{\cup_{x\,\in\,A}\{H_{x}:\,H_{x}\,\in\,\tau\}\}_{A\,\in\,\mathcal{P}}\}\,=\,\mathcal{B}$ satisfies:
\begin{itemize}
\item {$\mathcal{B}$ is an antichain with respect to the usual inclusion order.}
\item {If $A$ is a subset of $K$ not included in any element of $\mathcal{B}$, then there exists a two element subset of $A$ with the same property.}
\end{itemize}}

Then $\tau$ will be called an \emph{ortho-normal cover} of $K$ 
\end{enumerate}
\end{definition}

\begin{definition}
Let $\tau$ be a collection of subsets of an algebra $K\,=\,\left\langle\underline{K},\,f_{1},\,f_{2},\,\ldots \,f_{l} \right\rangle $ indexed by $K$, that satisfies
\begin{enumerate}
\item {$(\forall x \in K)(\exists \!y\,\in\,\tau)\,x\,\in\,y$}
\item {$\bigcup\,\tau\,=\,K$}
\item {$\tau$ is an antichain with respect to inclusion}
\item {For a not necessarily disjoint partition $\mathcal{P}$ of $K$,  $\{\cup_{x\,\in\,A}\{H_{x}:\,H_{x}\,\in\,\tau\}\}_{A\,\in\,\mathcal{P}}\}\,=\,\mathcal{B}$ satisfies:
\begin{itemize}
\item {$\mathcal{B}$ is an antichain with respect to the usual inclusion order.}
\item {If $A$ is a subset of $K$ not included in any element of $\mathcal{B}$, then there exists a two element subset of $A$ with the same property.}
\item {For any $f_{i}$ of arity $n$ and elements $B_{1},\ldots,B_{n}\,\in\,\mathcal{B}$ there exists an element $B\,\in\,\mathcal{B}$ such that $f(B_{1},\ldots,\,B_{n})\,\subseteq\,B$  }
\end{itemize}}

Then $\tau$ will be called an \emph{ortho-normal cover} of the algebra $K$ for the tolerance determined by $\mathcal{B}$. 
\end{enumerate}
\end{definition}

Note that $\mathcal{B}$ is a normal system of subsets of the algebra $K$ and therefore determines a unique compatible tolerance on $K$ (see \cite{CZ}) and conversely. The same thing happens in case of the first definition for the set $K$.  

\begin{definition}
Let the set of minimal elements in $\{\diamondsuit x\,:\,\xi x \,:\ml x\,\neq\,0\}$ of an abstract bitten algebra $S$  be $H_{0}$, then let $H\,=\,\{x\,:\,\diamondsuit x\,\in\,H_{0}\}$. If $H$ determines an ortho-normal cover on a TAS $P\,=\,\left\langle \underline{P},\,T\right\rangle$ with $Card (\underline{P})\,=\,Card(H)$ then $S$ will be said to be a \emph{refined abstract bitten algebra}. 
\end{definition}

\begin{theorem}\label{thm6}
For each refined abstract bitten algebra $S$ there exists a tolerance approximation space $K$, such that $Bite(K)\,\cong \,S$.
\end{theorem}

\begin{proof}
The proof has already been done above. The essential steps are:
\begin{enumerate}
\item {The definition of a refined abstract bitten algebra $S$, ensures the existence of a related TAS $P$ (say).}
\item {It can be checked that the TAS $P$ and the TAS $K$ (mentioned in the statement of the theorem) are isomorphic because of the representation theorem for tolerances.}
\item {Rest of the aspects have already been taken care of.}
\end{enumerate}
\end{proof}

Theorem~\ref{thm6} is hardly constructive in any sense of the term and may prove to be difficult to apply in particular situations.

\subsection{Discussion}

In the above, a semantics for the logic of roughly similar objects is developed and it has a certain relationship with the original TAS. However the actual level of relationship that is desired between \emph{such a semantics}, and an \emph{original generalized approximation space along with the associated process} is still the subject to some judgement. This is definitely independent of logic-forming strategies like the Gentzen style algebraic (\cite{CzJ}, \cite{CzP}) or abstract algebraic approach that can be applied.

Any bitten semantics with no restrictions on the type of granules can be expected to fall short of a unique representation theorem (in the sense that given the semantics, we have a specification for obtaining the original TAS in a unique way). We say this because in general, the process of forming approximations actually obscures the distribution of blocks. The latter is essential for a unique representation theorem because of $\cite{CZ}$.  

When the set of granules used is the set of $T$-related elements, the required conditions for a unique representation theorem will necessarily include a constructive instance of the following process of formation of blocks from sets of $T$-related elements.

\begin{itemize}
\item {Let $\mathcal{B}$ denote the set of all blocks of the TAS $S\,=\,\left\langle \underline{S},\,T\right\rangle$ and let $\tau\,=\,\{[x]_{T}\,:\,x\,\in,\underline{S}\}$. }
\item {Form the power set $\wp (\tau)$}
\item {Let $\mu (\tau)\,=\,\{\cup (K)\,:\,K\,\in\,\wp (\tau),\,T_{|\cup (K)}\;\mathrm{is}\;\mathrm{an}\;\mathrm{equivalence}\}$. $T_{|\cup (K)}$ being the restriction of the tolerance to the set $\cup (K)$.}
\item {$\mu (\tau)$ is partially ordered by the inclusion relation.}
\item {$\mu _{max}(\tau)$, the set of maximal elements of $\mu (\tau)$, is the set of blocks of $S$. That is, $\mu _{max}(\tau)\,=\,\mathcal{B}$.}
\end{itemize}

\subsection{Illustrative Example}

Let $S\,=\,\{x_{1}, x_{2}, x_{3}, x_{4}\}$ and let the tolerance $T$ be generated on it by $\{(x_{1}, x_{2}),\,(x_{2},x_{3})\}$. Taking the granules to be the set of $T$-related elements, we have $Gr(S)=\{(x_{1}:x_{2}),\,(x_{2}:x_{1},x_{3}),\,(x_{3}:x_{2}),\,(x_{4}:)\}$. Here $(x_{1}:x_{2})$ means the granule generated by $x_{1}$ is $(x_{1},\,x_{2})$. The different approximations are then as in the table below. 

The first column in the table is for keeping track of the elements in the quotient
$\wp (S) |\sim $ and can be used for checking the operations of Prop.~\ref{sec3}. The order structure is given by the Hasse diagram following the table. More details of the construction are omitted because the next step will take some space.

\begin{center}
\begin{tabular}{|c|c|c|c|c|c|c|}
\hline\hline
 $\mathbf{\wp (S) |\sim}$  &  \textbf{Subset} & $\mathbf{X}$ & $\mathbf{Gr_{*}(X)}$ & $\mathbf{Gr^{*}(X)}$ & $\mathbf{Gr_{*}(X^{c})}$ & $\mathbf{Gr_{b}^{*}(X)}$ \\
\hline 
$B_{1}$ & $A_{1}$ & $\{x_{1}\}$ & $\emptyset$ & $\{x_{2},\,x_{1}\}$ & $\{x_{2},x_{3}, x_{4}\}$ & $\{x_{1}\}$ \\
\hline
$B_{2}$ & $A_{2}$ & $\{x_{2}\}$ & $\emptyset$ & $\{x_{1}, x_{2}, x_{3}\}$ & $\{x_{4}\}$ & $\{x_{1}, x_{2}, x_{3}\}$ \\
\hline
$B_{3}$ & $A_{3}$ & $\{x_{3}\}$ & $\emptyset$ & $\{x_{1}, x_{2}, x_{3}\}$ & $\{x_{1}, x_{2}, x_{4}\}$ & $\{x_{3}\}$ \\
\hline
$B_{4}$ & $A_{4}$ & $\{x_{4}\}$ & $\{x_{4}\}$ & $\{x_{4}\}$ & $\{x_{1}, x_{2}, x_{3}\}$ & $\{x_{4}\}$ \\
\hline
$B_{5}$ & $A_{5}$ & $\{x_{1}, x_{2}\}$ & $\{x_{1}, x_{2}\}$ & $\{x_{1}, x_{2}, x_{3}\}$ & $\{x_{4}\}$ & $\{x_{1}, x_{2}, x_{3}\}$ \\
\hline

$B_{2}$ & $A_{6}$ & $\{x_{1}, x_{3}\}$ & $\emptyset$ & $\{x_{1}, x_{2}, x_{3}\}$ & $\{x_{4}\}$ & $\{x_{1}, x_{2}, x_{3}\}$ \\
\hline
$B_{6}$ & $A_{7}$ & $\{x_{1}, x_{4}\}$ & $\{x_{4}\}$ & $S$ & $\{x_{2}, x_{3}\}$ & $\{x_{1}, x_{4}\}$ \\
\hline
$B_{7}$ & $A_{8}$ & $\{x_{2}, x_{3}\}$ & $\{x_{2}, x_{3}\}$ & $\{x_{1}, x_{2}, x_{3}\}$ & $\{x_{4}\}$ & $\{x_{1}, x_{2}, x_{3}\}$ \\
\hline
$B_{8}$ & $A_{9}$ & $\{x_{2}, x_{4}\}$ & $\{x_{4}\}$ & $S$ & $\emptyset$ & $S$ \\
\hline
$B_{9}$ & $A_{10}$ & $\{x_{3}, x_{4}\}$ & $\{x_{4}\}$ & $S$ & $\{x_{1}, x_{2}\}$ & $\{x_{3}, x_{4}\}$ \\
\hline

$B_{10}$ & $A_{11}$ & $\{x_{1}, x_{2}, x_{3}\}$ & $\{x_{1}, x_{2}, x_{3}\}$ & $\{x_{1}, x_{2}, x_{3}\}$ & $\{x_{4}\}$ & $\{x_{1}, x_{2}, x_{3}\}$ \\
\hline
$B_{11}$ & $A_{12}$ & $\{x_{1}, x_{2}, x_{4}\}$ & $\{x_{1}, x_{2}, x_{4}\}$ & $S$ & $\emptyset$ & $S$ \\
\hline
$B_{12}$ & $A_{13}$ & $\{x_{2}, x_{3}, x_{4}\}$ & $\{x_{2}, x_{3}, x_{4}\}$ & $S$ & $\emptyset$ & $S$ \\
\hline
$B_{8}$ & $A_{14}$ & $\{x_{1}, x_{3}, x_{4}\}$ & $\{x_{4}\}$ & $S$ & $\emptyset$ & $S$ \\
\hline
$B_{14}$ & $A_{15}$ & $S$ & $S$ & $S$ & $\emptyset$ & $S$ \\
\hline

$B_{13}$ & $A_{16}$ & $\emptyset$ & $\emptyset$ & $\emptyset$ & $S$ & $\emptyset$ \\
\hline
\end{tabular}
\end{center}

\includegraphics{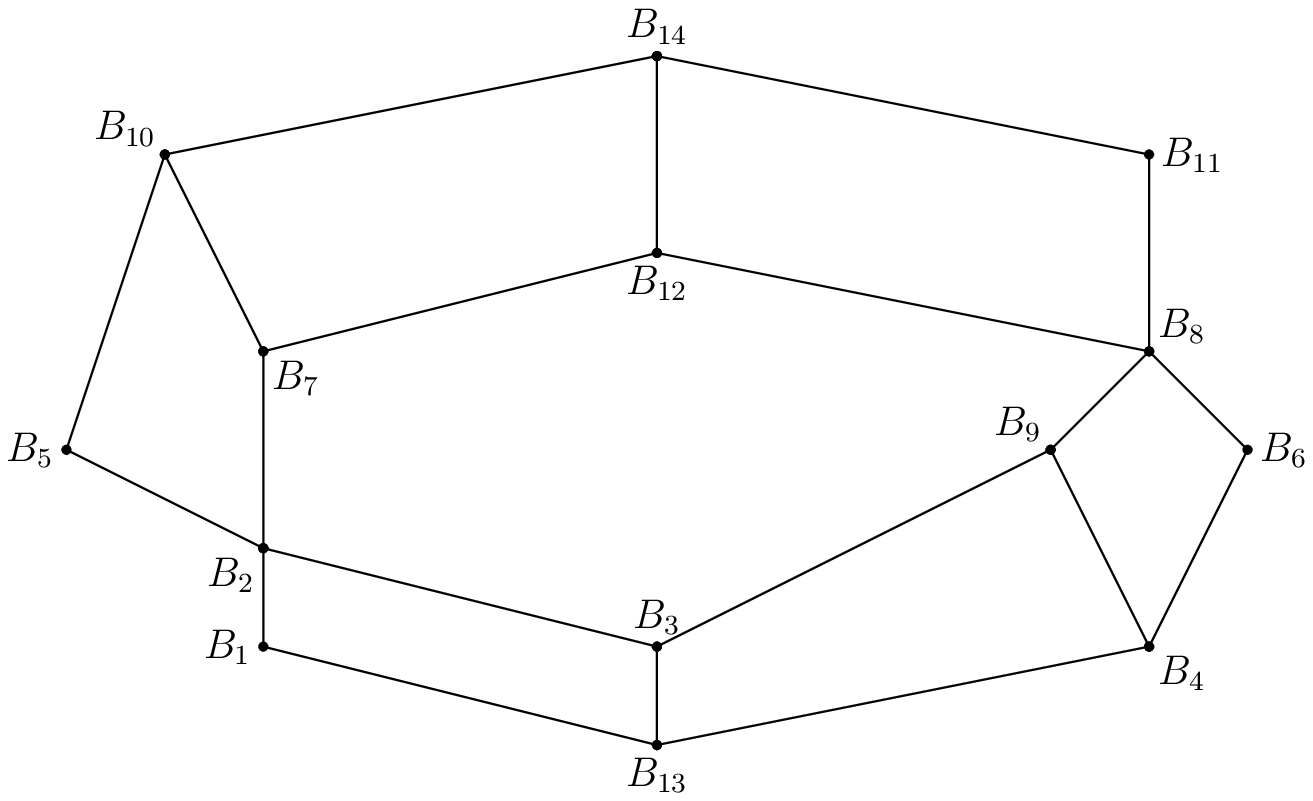}
\begin{center}
\textbf{Fig.1: Partial Order on the Quotient Partial Algebra} 
\end{center}

\hrulefill

\subsection*{Problems}
Some important problems that originate from the previous sections are:
\begin{enumerate}
\item {Under what general conditions will the operation of biting make the partial operation in Prop. ~\ref{sec3} total?}
\item {Find simpler conditions under which an abstract bitten algebra becomes a refined abstract bitten algebra.}
\item {Describe the quasi equational classes (and variants) of refined abstract bitten algebras }
\item {Does a complete, atomic double Heyting algebra determine a unique BAS? }
\item {Which type of can reducts be computed with the help of these algebras?}
\end{enumerate}

\section{Alternative Approach}

The appropriate semantic domain for the above semantics may be considered by some to be less natural on subjective grounds. Among these, the difficulty involved in reasoning within the power set of the set of possible order-compatible partitions of the set of roughly equivalent elements may be cited. We introduce a simpler semantics for these reasons, and also because of possible connections with the semantics for the choice inclusive similarity based rough set theory due to the present author in \cite{AM69} (we consider extensions of the same in two forthcoming papers). 

There are many differences between the use of choice functions in \cite{AM69} and in the present paper. In the former, the use of blocks as granules permit an elegant interpretation in the context of \emph{local clear discernibility} and generalizations of Pawlak's theory of knowledge. This paradigm needs a complete reworking in the present context as blocks do not always relate nicely to the concept of granules used. Moreover choice is used in the definition of approximations. So we avoid doing these in the present paper and restrict ourselves to developing a semantics with less commitment to intended meaning. Choice functions are used in defining the rough operations of \emph{combining sets} and \emph{extracting the common part of two sets}. 

\begin{definition}
For any $a,\,b\,\in\,\wp(S)|\sim$, let $UB(a,\,b)$ and $LB(a,\,b)$ be the set of minimal upper bounds and the set of maximal lower bounds of $a$ and $b$ (we assume that these are nonempty for all pairs $a,\,b$). If $\lambda :\,\wp(\wp (S)|\sim )\,\longmapsto\,\wp (S)|\sim$ is a choice function, (by definition, it is such that $(a\,\leq\,b\,\longrightarrow\,\lambda(UB(\{a,\,b\}))\,=\,b,\,\lambda(LB(\{a,\,b\}))\,=\,a)$), then let 
\begin{itemize}
 \item {$a\,+\,b\,=\,\lambda(UB(\{a,\,b\}))$}
 \item {$a\,\cdot\,b\,=\,\lambda(LB(\{a,\,b\}))$.}
\end{itemize}

$\mathfrak{S}\,=\,\left\langle \wp (S)|\sim,\,+\,\cdot,\,L,\,\blacklozenge,\,\neg\right\rangle $ will be called the \emph{simplified algebra of the bitten granular semantics} (\textsf{SGBA})
\end{definition}

\begin{theorem} A \textsf{SGBA}, $\mathfrak{B}\,=\,\left\langle \underline{B},\,+\,\cdot,\,L,\,\blacklozenge,\,\neg\right\rangle $ satisfies all of the following: 
\begin{enumerate}
\item {$\left\langle \underline{B},\,+\,\cdot \right\rangle$ is a $\lambda$-lattice}
\item {$a\,+\,b\,=\,b\,+\,a$;   $\;a\,\cdot\,b\,=\,b\,\cdot\,a$}
\item {$a\,+\,a\,=\,a $;   $\;a\,\cdot\,a\,=\,a$}
\item {$a\,+\,(a\,\cdot\,b)\,=\,a $;   $\;a\,\cdot\,(a\,+\,b)\,=\,a $}
\item {$a\,+\,(a\,+\,(b\,+\,c))\,=\,a\,+\,(b\,+\,c) $;   $\;a\,\cdot\,(a\,\cdot\,(b\,\cdot\,c))\,=\,a\,\cdot\,(b\,\cdot\,c) $}
\item {$a\,+\,La\,=\,a$;   $\;a\,\cdot\,La\,=\,La $}
\item {$a\,+\,\blacklozenge a\,=\,\blacklozenge a $;   $\; a\,\cdot\,\blacklozenge a\,=\,a $}
\item {$L(La)\,=\,La $;   $\; \blacklozenge(\blacklozenge a)\,=\,\blacklozenge a $}
\item {$(a\,+\,b\,=\,a\,\longrightarrow\,La\,+\,Lb\,=\,La) $;   $\; (a\,\cdot\,b\,=\,a\,\longrightarrow\,La\,\cdot\,Lb\,=\,La) $}
\item {$(a\,+\,b\,=\,a\,\longrightarrow\,\blacklozenge a\,+\,\blacklozenge b\,=\,\blacklozenge a) $;   $\; (a\,\cdot\,b\,=\,a\,\longrightarrow\,\blacklozenge a\,\cdot\,\blacklozenge b\,=\,\blacklozenge a) $}
\item {$\neg(La)\,=\,\blacklozenge (\neg a) $;   $\; \neg (\blacklozenge a)\,=\,L(\neg a) $}
\item {$L0\,=\,0,\;L1\,=\,1 $;   $\; \blacklozenge 0\,=\,0,\; \blacklozenge 1\,=\,1$}
\item {$L \blacklozenge a \,+\,\blacklozenge a\,=\,\blacklozenge a $;   $\; L \blacklozenge a \,\cdot\,\blacklozenge a\,=\,L \blacklozenge a$}
\item {$L a\,+\,\blacklozenge L a\,=\,\blacklozenge L a $;   $\; L a\,\cdot\,\blacklozenge L a\,=\, L a$}
\end{enumerate}
\end{theorem}

\begin{proof}
The proof consists in verification and is not very hard. 
\end{proof}

\subsection{Representation Problem for SGBAs}

Given a TAS $S$ in a bitten rough semantic perspective, associating a single \textsf{SGBA} as its corresponding semantics amounts to modifying the original meaning by the introduction of artificial choice functions for the purpose of forming rough union and intersection-like operations. Either we need a justification of such preference or accept all of the possible preferences. So the default semantics must be given by a set of \textsf{SGBA}s indexed by the set of all possible choice functions in the lambda lattice formation context. In this perspective the semantics can be explained directly and a sequent calculus associated (and with little additional representation theory).

Let $x\,\in\,S$, then $([x]_{T})^{l}\,=\,[x]_{T} $, while $([x]_{T})^{u}_{b}\,=\,([x]_{T})^{u}\,\setminus\,(S\,\setminus\,[x]_{T})^{l})$. It is obvious that elements with nonempty lower approximation that are minimal with respect to the rough order will be equivalent to elements of this type. Once the order relation on the set of roughly equivalent elements has been deduced, then we can find the elements of this type. This permits the reconstruction of the equivalence based partition of  the power set of $S$. 

If we do not know all the choice functions involved, then it is not possible in general to determine or construct the blocks of the tolerance $T$. But when will a knowledge of given subsets of choice functions permit us to determine the blocks of $T$? This is the problem of representation of \textsf{SGBA}s. It is also significant in a more general algebraic setting.

\section{Connections with AUAI Approximation Systems}

In a $AUAI$ approximation system $\left\langle S,\,\mathcal{K} \right\rangle $, the collection $\mathcal{K}$ need not be the most appropriate concept of granule for the four different approximations of the theory (this is considered in detail in a forthcoming paper by the present author). The implicit conditions on the possible concepts of a granule in the bitten approach are the following:
\begin{itemize}
\item {The set of granules $\mathcal{S}$ is a partition of $S$, that is $\bigcup \mathcal{S}\,=\,S$.}
\item {The form of the lower and bitten upper approximation are given as in the subsection on 'Bitten Approach'}
\end{itemize}

\begin{theorem}
Given a \textsf{BAS} $\left\langle S,\,Gr(S),\,T,\,Gr_{*},\,Gr_{b}^{*}\right\rangle $, the $AUAI$ approximation system $\left\langle S,\,Gr(S) \right\rangle $ satisfies 
\begin{enumerate}
\item {$(\forall {X}\,\in\,\wp(S))\, X^{l1}\,=\,Gr_{*}(X)$}
\item {$(\forall {X}\,\in\,\wp(S))\, X^{u1}\,\subseteq\,Gr^{*}(X)$}
\item {$(\forall {X}\,\in\,\wp(S))\, Gr_{b}^{*}(X)\,=\,Gr^{*}(X)\,\cap\,X^{u2} $}
\end{enumerate}
\end{theorem}

\begin{proof}
\begin{enumerate}
\item {$X^{l1}\,=\,\bigcup\{A:\,A\,\subseteq\,X\, ;\,A\,\in\,Gr(S) \}\,=\,Gr_{*}(X)$}
\item {$X^{u1}$ is the intersection of all unions of elements of $Gr(S)$, while $Gr^{*}(X)$ is the union of all elements of $Gr(S)$ that have non empty intersection with $X$. In general if $Gr(S)$ is a collection of pairwise disjoint sets then $X^{u1}\,=\,Gr^{*}(X)$, else $X^{u1}\,\subseteq\,Gr^{*}(X)$.}
\item {$Gr_{b}^{*}(X)\,=\,Gr^{*}(X)\,\setminus\, Gr_{*}(X^c)\,=\,Gr^{*}(X)\,\cap\,(Gr_{*}(X^c))^{c}$. But $\bigcap \{A_{i}^{c}:\,X\,\subseteq\,A_{i}^{c}\,;\, A_{i}\in\,Gr(S)\}\,=\,\bigcap \{A_{i}^{c}:\,X^{c}\,\subseteq\,A_{i}\,\in\,Gr(S)\}\,=\,(\bigcup\{A_{i}:\,X^{c}\,\subseteq\,A_{i}\})^{c}\,=\,(Gr_{*}(X^{c}))^{c}\,=\,X^{u2}$. So $Gr_{b}^{*}(X)\,=\,Gr^{*}(X)\,\cap\,X^{u2}$ holds.}
\end{enumerate}
\end{proof}

In the above theorem, we have taken the collection $\mathcal{K}$ of the $AUAI$ approximation system $\left\langle S,\,\mathcal{K} \right\rangle $ to coincide with Gr(S). This need not be the case in general and many variations are possible on the point. In particular we can select the collection $\mathcal{K}$ so that $Gr_{b}^{*}(X)$ coincides with $X^{u2}$.

\section{Concluding Remarks}

In this research paper, we have developed two different algebraic semantics of bitten rough set theory. Topology is involved in a explicit way in the first of two. In the bitten approach, some types of granules can hinder possible seamless representation theorems, while others may be more useful. So the theory is not independent of the type of granules in entirety. The associated logics can be expected to have high expressive power due to the higher order approach used.

A positive solution for the first problem would mean a much simpler algebraic semantics provided we can fix the notion of logical consequence in a suitable way.  
But even in those cases the first semantics would remain relevant.

The nature of the semantic domain used and therefore the effective object level in the first approach is very different and amounts to a new paradigm in rough set theory. In the second approach, though the semantic domain is natural, it is not a fully explored one in the context of rough sets (\cite{AM69}). If we use different types of semantic domain at the object level, then some natural relative distortions may creep in. From the technical point of view this may or may not affect actual applications. A deeper study of such distortions will be considered in future work.

The second approach can be accordant with different interpretations of choice and therefore is a more open ended semantics. Application of this type of semantics requires a more conscious and regulated way of forming approximations or 'specifying the indiscernibles'.

The algebraization strategy developed is relevant for other types of generalized rough set theory. These include generalized cover and pure granule based approaches. For these reasons, we will consider these connections and sequent calculi for the algebraic semantics in a separate paper. 

\bibliographystyle{fundam}
\bibliography{newsem33.bib}
\end{document}